\documentclass[a4paper,12pt]{amsart}

\usepackage[latin2]{inputenc}
\usepackage{longtable}
\usepackage{leftidx}
\usepackage{amsmath}
\usepackage{amscd,amssymb,amsthm}

\begin{document}

\title{Summation of Hyperharmonic Series}
\author{Istv\'an Mez\H{o}}
\address{Institute of Mathematics, University of Debrecen, Hungary}
\email{imezo@math.klte.hu}
\urladdr{http://www.math.klte.hu/algebra/mezo.htm}

\newcommand{\NN}{\mathbb{N}}
\newcommand{\ZZ}{\mathbb{Z}}
\newcommand{\QQ}{\mathbb{Q}}

\newcommand {\Li}{\mathop{\textup{Li}}\nolimits}
\newcommand{\stirling}[2]{\genfrac[]{0pt}{}{#1}{#2}}

\newtheorem{Theorem}{Theorem}
\newtheorem{Lemma}[Theorem]{Lemma}
\newtheorem{Corollary}[Theorem]{Corollary}
\theoremstyle{definition}
\newtheorem{Definition}[Theorem]{Definition}
\newtheorem{Example}[Theorem]{Example}

\begin{abstract}
We shall show that the sum of the series formed by the so-called hyperharmonic numbers can be expressed in terms of the Riemann zeta function. More exactly, we give summation formula for the general hyperharmonic series.
\end{abstract}

\maketitle

\section{Hyperharmonic numbers}

\textbf{Introduction.} In 1996, J. H. Conway and R. K. Guy in \cite{CG} have defined the notion of hyperharmonic numbers.

The $n$-th harmonic number is the $n$-th partial sum of the harmonic series:
\[H_n=\sum_{k=1}^n\frac{1}{k}.\]

$H_n^{(1)}:=H_n$, and for all $r>1$ let
\[H_n^{(r)}=\sum_{k=1}^n H_k^{(r-1)}\]
be the $n$-th hyperharmonic number of order $r$.
These numbers can be expressed by binomial coefficients and ordinary harmonic numbers:
\begin{eqnarray}
H_n^{(r)}=\binom{n+r-1}{r-1}(H_{n+r-1}-H_{r-1}).\label{Hnr_binom}
\end{eqnarray}

It turned out that the hyperharmonic numbers have many combinatorial connections. To
present this fact, we need to introduce the notion of $r$-Stirling numbers. Getting deeper insight, see \cite{BGG} and the references given there.

\textbf{$r$-Stirling numbers.} $\stirling{n}{k}_r$ is the number of permutations of the set $\{1,\dots,n\}$ having $k$ disjoint,
non-empty cycles, in which the elements $1$ through $r$ are restricted to appear in different cycles.

The following identity integrates the hyperharmonic- and the $r$-Stirling numbers.
\[\frac{\stirling{n+r}{r+1}_r}{n!}=H_n^{(r)}.\]
This equality will be used in the special case $r=1$ \cite{GKP}:
\begin{equation}
\frac{\stirling{n+1}{2}}{n!}:=\frac{\stirling{n+1}{2}_r}{n!}=H_n.\label{stirling_Hn}
\end{equation}

\section{Results up to the present}

Our goal is to determine the sum of the series has the form
\[\sum_{n=1}^\infty\frac{H_n^{(r)}}{n^m},\]
for all $r\ge 2$ and possible $m$.

To do this, we should determine the asymptotic behaviour of hyperharmonic numbers. In the paper \cite{M} there is a skimped approximation which helps us to get convergence theorems for hyperharmonic series. Namely, we have that
\[\frac{1}{r!}n^{r-1}<H_n^{(r)}<\frac{3}{2}\frac{(2r)^r}{(r-1)!}n^r,\]
for all $n\in\NN$ and $r\ge 2$.
According to this result, the followings were proved:
\begin{align*}
&\sum_{n=1}^\infty\frac{H_n^{(r)}}{n^r}=\infty,\\
\frac{\zeta(s+1)}{r!}<&\sum_{n=1}^\infty\frac{H_n^{(r)}}{n^{r+s}}<\frac{3}{2}\frac{(2r)^r}{(r-1)!}\zeta(s)\quad(s>1),\\
&\sum_{n=1}^\infty\frac{n^{r-1}}{H_n^{(r)}}=\infty,\\
\frac{2}{3}\frac{(r-1)!}{(2r)^r}\zeta(s)<&\sum_{n=1}^\infty\frac{n^{r-s}}{H_n^{(r)}}<r!\zeta(s-1)\quad(s>2),
\end{align*}
where $\zeta$ is the Riemann zeta function and $r\ge 2$.

\section{Asymptotic approximation}

To have the exact asymptotic behaviour of hyperharmonic numbers we need the following inequality from \cite{CG}.
\begin{equation}
\frac{1}{2(n+1)}+\ln(n)+\gamma<H_n<\frac{1}{2n}+\ln(n)+\gamma\quad(n\in\NN),\label{Hn_estim}
\end{equation}
where $\gamma=0.5772\dots$ is the Euler-Mascheroni constant.

Now we formulate our first result.
\begin{Lemma}For all $n\in\NN$ and for a fixed order $r\ge 2$ we have
\[H_n^{(r)}\sim\frac{1}{(r-1)!}\left(n^{r-1}\ln(n)\right),\]
that is, the quotient of the left and right hand side tends to 1.
\end{Lemma}

\begin{proof}The binomial coefficient in \eqref{Hnr_binom} has the form
\[\binom{n+r-1}{r-1}=\frac{(n+r-1)!}{(r-1)!n!}=\frac{1}{(r-1)!}(n+1)(n+2)\cdots(n+r-1).\]
It means that for a fixed order $r$
\[\binom{n+r-1}{r-1}\sim\frac{1}{(r-1)!}n^{r-1}.\]
For the convenience let us introduce the abbreviation $t:=r-1$. We should estimate the factor $H_{n+t}-H_t$ in \eqref{Hnr_binom}. According to \eqref{Hn_estim}, we get that
\[H_{n+t}-H_t<\frac{1}{2(n+t)}+\ln(n+t)+\gamma-\frac{1}{2(t+1)}-\ln(t)-\gamma.\]
Since $\frac{1}{2(n+t)}<\frac{1}{2(1+t)}$ and $\ln(n+t)-\ln(t)=\ln\left(\frac{n+t}{t}\right)$, so
\[H_{n+t}-H_t<\ln(n+t).\]
The lower estimation can be deduced as follows
\[H_{n+t}-H_t>\frac{1}{2(n+t+1)}+\ln(n+t)+\gamma-\frac{1}{2t}-\ln(t)-\gamma>\]
\[>\ln(n+t)-\ln(t)-\frac{1}{2}=\ln(n+t)-\ln(t\sqrt{e}).\]
From these we get that
\[\ln(n+t)-\ln(t\sqrt{e})<H_{n+t}-H_t<\ln(n+t),\]
whence
\[1-\frac{\ln(t\sqrt{e})}{\ln(n+t)}<\frac{H_{n+t}-H_t}{\ln(n+t)}<1.\]
The limit of the left-hand side formula is $1$ as $n$ tends to infinity. Therefore (remember that $t=r-1$)
\[H_{n+r-1}-H_{r-1}\sim\ln(n+r-1)\sim\ln(n).\]
Collecting the results above we get the statement of the Lemma.
\end{proof}

\begin{Corollary}The following series are convergent
\[\sum_{n=1}^\infty\frac{H_n^{(r)}}{n^m}<+\infty,\]
whenever $m>r$.
\end{Corollary}

\begin{proof}Because of Lemma 1,
\[\sum_{n=1}^\infty\frac{H_n^{(r)}}{n^m}<C\sum_{n=1}^\infty\frac{n^{r-1}\ln(n)}{n^m}<C\sum_{n=1}^\infty\frac{\ln(n)}{n^2}<+\infty\]
under the assumption $m\ge r+1$.
\end{proof}

\section{Generating functions, Euler sums and Hypergeometric Series}
In this section we introduce the notions needed in the proof.

\textbf{Generating functions.} Let $(a_n)_{n\in\NN}$ be a real sequence. Then the function
\[f(z):=\sum_{n=0}^\infty a_nz^n\]
is called the generating function of $(a_n)_{n\in\NN}$.
If $a_n=H_n$ we get that (see \cite{GKP,BGG})
\[\sum_{n=0}^\infty H_nz^n=-\frac{\ln(1-z)}{1-z},\]
and in general
\begin{equation}
\sum_{n=0}^\infty H_n^{(r)}z^n=-\frac{\ln(1-z)}{(1-z)^r}.\label{genfunc_Hnr}
\end{equation}
The generating function
\begin{equation}
\frac{1}{m!}\left(-\ln(1-z)\right)^m=\sum_{n=1}^\infty\stirling{n}{m}\frac{z^n}{n!}\label{genfunc_stirling}
\end{equation}
can be found in \cite{GKP,B}.

The well known polylogarithm functions can also be considered as generating functions belong to $a_n=\frac{1}{n^k}$ (for a fixed $k$).
\[\Li_k(z)=\sum_{n=1}^\infty\frac{z^n}{n^k}\quad(k=1,2,\dots).\]

The last remarkable function needed by us is
\[\frac{1}{(1-z)^k}=\sum_{n=0}^\infty\binom{n+k-1}{n}z^n.\]

\textbf{Euler sums.} The general Euler sum is an infinite sum whose general term is a product of harmonic numbers divided by some power of $n$, see the comprehensive paper \cite{FS}. The sum
\[\sum_{n=1}^\infty\frac{H_n}{n^m}=\frac{1}{2}(m+2)\zeta(m+1)-\sum_{k=1}^{m-2}\zeta(m-k)\zeta(k+1)\]
was derived by Euler (see \cite{BB} and the references given there). Related series were studied by De Doelder in \cite{dD} and Shen \cite{S}, for instance.

\textbf{Hypergeometric series.} The Pochhammer symbol is defined by the formula
\begin{equation}
(x)_n=x(x+1)\cdots(x+n-1),\label{pochhammer}
\end{equation}
with special cases $(1)_n=n!$ and $(x)_1=x$. The definition of the hypergeometric function (or hypergeometric series) is the following:
\[\leftidx{_n}{F}{_m}\left(\left.\begin{tabular}{llll}$a_1,$&$a_2,$&$\dots,$&$a_n$\\$b_1,$&$b_2$,&$\dots,$&$b_m$\end{tabular}\right|z\right)=\sum_{k=0}^\infty\frac{(a_1)_k(a_2)_k\cdots(a_n)_k}{(b_1)_k(b_2)_k\cdots(b_m)_k}\frac{z^k}{k!}.\]
This function will appear in the sum of the hyperharmonic numbers.

We shall need one more statement.
\begin{Lemma}We have
\[\int\frac{\ln(z)}{(1-z)z}dt=\Li_2(1-z)+\frac{1}{2}\ln^2(z),\]
and for all $2\le r\in\NN$
\[\int\frac{\ln(z)}{(1-z)z^r}dz=\int\frac{\ln(z)}{(1-z)z^{r-1}}dz-\frac{\ln(z)}{(r-1)z^{r-1}}-\frac{1}{(r-1)^2z^{r-1}},\]
or, equivalently,
\[\int\frac{\ln(z)}{(1-z)z^r}dz=\Li_2(1-z)+\frac{1}{2}\ln^2(z)-\sum_{k=1}^{r-1}\left(\frac{\ln(z)}{kz^k}+\frac{1}{k^2z^k}\right).\]
up to additive constants
\end{Lemma}

\begin{proof}
The definition of $\Li_2$ readily gives that
\[\Li_2'(1-z)=\frac{\ln(z)}{1-z}.\]
Moreover,
\[\left[\frac{1}{2}\ln^2(z)\right]'=\frac{\ln(z)}{z},\]
whence
\[\Li_2'(1-z)+\left[\frac{1}{2}\ln^2(z)\right]'=\frac{z\ln(z)+(1-z)\ln(z)}{(1-z)z}=\frac{\ln(z)}{(1-z)z}.\]
The first statement is proved. The second one also can be deduced by differentiation. The derivative of the right-hand side has the form
\[\frac{\ln(z)}{(1-z)z^{r-1}}-\frac{(r-1)z^{r-2}-(r-1)^2\ln(z)z^{r-2}}{(r-1)^2(z^{r-1})^2}-\frac{-(r-1)}{(r-1)^2z^r}=\]
\[=\frac{z\ln(z)}{(1-z)z^r}-\frac{z^r-(r-1)\ln(z)z^r}{(r-1)(z^r)^2}+\frac{1}{(r-1)z^r}=\]
\[=\frac{z^{r+1}\ln(z)(r-1)-z^r(1-z)+(1-z)(r-1)\ln(z)z^r+z^r(1-z)}{(r-1)z^{2r}(1-z)}=\]
\[=\frac{z\ln(z)+(1-z)\ln(z)}{z^r(1-z)}=\frac{\ln(z)}{z^r(1-z)},\]
as we want.
\end{proof}

\section{The summation formula}

For the sake of simplicity, we introduce the notations
\[S(r,m):=\sum_{n=1}^\infty\frac{H_n^{(r)}}{n^m},\]
and
\[B(k,m):=\leftidx{_{m+1}}{F}{_m}\left(\left.\begin{tabular}{lllll}$1,$&$1,$&$\dots,$&$1,$&$k+1$\\$2,$&$2$,&$\dots,$&$2$\end{tabular}\right|1\right).\]

After these introductory steps we are ready to prove the main theorem.

\begin{Theorem} If $r\ge 2$ and $m\ge r+1$, then
\[S(r,m)=S(1,m)+\sum_{k=1}^{r-1}\frac{1}{k}\left[S(k,m-1)-B(k,m)\right].\]
\end{Theorem}

\begin{proof}We begin with the generating function \eqref{genfunc_Hnr}. Division with $z$ and integration gives that
\[\sum_{n=1}^\infty \frac{H_n^{(r)}}{n}z^n=-\int\frac{\ln(1-z)}{z(1-z)^r}dz.\]
An integral transformation gives that
\[-\int\frac{\ln(1-z)}{z(1-z)^r}dz=\int\frac{\ln(z)}{(1-z)z^r}dz=\]
\[=\Li_2(z)+\frac{1}{2}\ln^2(1-z)-\sum_{k=1}^{r-1}\left(\frac{\ln(1-z)}{k(1-z)^k}+\frac{1}{k^2(1-z)^k}\right).\]
According to \eqref{genfunc_Hnr} and \eqref{genfunc_stirling} one can write
\[\sum_{n=1}^\infty \frac{H_n^{(r)}}{n}z^n=\Li_2(z)+\sum_{n=1}^\infty\stirling{n}{2}\frac{z^n}{n!}-\]
\[\sum_{k=1}^{r-1}\left(\frac{1}{k}(-1)\sum_{n=0}^\infty H_n^{(r)}z^n+\frac{1}{k^2}\sum_{n=0}^\infty\binom{n+k-1}{n}z^n\right).\]
It is obvious that this series is divergent but the generating functions works well without any restriction.

Let us deal with the second term. The Stirling numbers satisfy the recurrence relation
\[\stirling{n}{k}=(n-1)\stirling{n-1}{k}+\stirling{n-1}{k-1}\quad(n>0).\]
From this
\[\stirling{n+1}{2}=n\stirling{n}{2}+\stirling{n}{1}=n\stirling{n}{2}+(n-1)!\quad(n>0).\]
Now, \eqref{stirling_Hn} can be rewritten as follows
\[H_n=\frac{1}{n!}\stirling{n+1}{2}=\frac{1}{(n-1)!}\stirling{n}{2}+\frac{1}{n}.\]
Division with $n$ and rearrangement give that
\[\frac{1}{n!}\stirling{n}{2}=\frac{H_n}{n}-\frac{1}{n^2}.\]
Therefore the second sum is
\[\sum_{n=1}^\infty\stirling{n}{2}\frac{z^n}{n!}=\sum_{n=1}^\infty\frac{H_n}{n}z^n-\sum_{n=1}^\infty\frac{z^n}{n^2}.\]
Since the last member equals to $\Li_2(z)$, it cancels the first member of the sum above. Hence
\[\sum_{n=1}^\infty \frac{H_n^{(r)}}{n}z^n=\sum_{n=1}^\infty\frac{H_n}{n}z^n+\sum_{k=1}^{r-1}\left(\frac{1}{k}\sum_{n=0}^\infty H_n^{(r)}z^n-\frac{1}{k^2}\sum_{n=0}^\infty\binom{n+k-1}{n}z^n\right).\]
An easy induction shows that (after dividing with $z$, integrating, and repeating these steps $(m-1)$-times and finally substituting $z=1$)
\begin{equation}
\sum_{n=1}^\infty \frac{H_n^{(r)}}{n^m}=S(1,m)+\sum_{k=1}^{r-1}\left(\frac{1}{k}S(r,m-1)-\frac{1}{k^2}\sum_{n=1}^\infty\binom{n+k-1}{n}\frac{1}{n^{m-1}}\right).\label{prelim}
\end{equation}
The last step is the transformation of the last member.
\[\sum_{n=1}^\infty\binom{n+k-1}{n}\frac{1}{n^{m-1}}=\sum_{n=1}^\infty\frac{(n+k-1)!}{n!(k-1)!}\frac{1}{n^{m-1}}=\]
\[\frac{1}{(k-1)!}\sum_{n=1}^\infty(n+1)(n+2)\cdots(n+k-1)\frac{1}{n^{m-1}}=\]
\[\frac{1}{(k-1)!}\sum_{n=1}^\infty\frac{(n)_k}{n^m},\]
because of the definition of the Pochhammer symbol in formula \eqref{pochhammer}.
On the other hand, the definition of $B(k,m)$ yields that
\[B(k,m)=\sum_{n=0}^\infty\frac{(n!)^m}{(n+1)!^m}\frac{(k+1)_n}{n!}=\sum_{n=0}^\infty\frac{1}{(n+1)^m}\frac{(k+1)_n}{n!}.\]
The next conversion should be applied:
\[\frac{k!(k+1)_n}{n!}=\frac{(k+n)!}{n!}=(n+1)(n+2)\cdots(n+k)=(n+1)_k.\]
It means that the equality
\begin{equation}
B(k,m)=\frac{1}{k!}\sum_{n=0}^\infty\frac{(n+1)_k}{(n+1)^m}=\frac{1}{k!}\sum_{n=1}^\infty\frac{(n)_k}{n^m}\label{Bkm}
\end{equation}
holds. That is,
\[\sum_{n=1}^\infty\binom{n+k-1}{n}\frac{1}{n^{m-1}}=kB(k,m).\]
Considering this and \eqref{prelim} the result follows.
\end{proof}

\section{Tabular of the low-order sums}

In the following tabulars we collect the low-order results of the Summation Theorem. We used the following identities which can be easily derived from \eqref{Bkm} and \eqref{pochhammer}.
\begin{align*}
B(1,m)&=\zeta(m-1),\\
B(2,m)&=\frac{1}{2}\left(\zeta(m-1)+\zeta(m-2)\right),\\
B(3,m)&=\frac{1}{6}\zeta(m-3)+\frac{1}{2}\zeta(m-2)+\frac{1}{3}\zeta(m-1).
\end{align*}

A computation with the mathematical package Maple shows that an improved accuracy can be reached using Theorem 1, in spite of calculating the series term-by-term. Namely, the sum
\[\sum_{n=1}^{10^5}\frac{H_n^{(4)}}{n^5}=1.310972037,\]
and takes about 500 seconds on an average personal computer, while the "closed form" shown below gives immediately the closer value
\[\sum_{n=1}^\infty\frac{H_n^{(4)}}{n^5}=\frac{\pi^6}{540}-\frac{\pi^4}{810}-\frac{11\pi^2}{216}-\zeta(3)-\frac{11\pi^2}{36}\zeta(3)-\frac{1}{2}\zeta(3)^2+\frac{11}{2}\zeta(5)\]
\[\approx 1.310990854,\]
with ten digits accuracy.

\begin{center}
$S(2,m)$\\
\begin{longtable}{|l|p{60mm}|l|l|}
\hline Power of $n$&Closed form&Approx. value\\\hline
$m=3$&$\frac{\pi^4}{72}-\frac{\pi^2}{6}+2\zeta(3)$&2.112083781\\\hline
$m=4$&$\frac{\pi^4}{72}+3\zeta(5)-\zeta(3)\left(1+\frac{\pi^2}{6}\right)$&1.284326055\\\hline
$m=5$&$\frac{\pi^6}{540}-\frac{\pi^4}{90}-\frac{1}{2}\zeta(3)^2+3\zeta(5)-\frac{\pi^2}{6}\zeta(3)$&1.109035642\\\hline
$m=6$&$\frac{\pi^6}{540}+4\zeta(7)-\frac{\pi^4}{90}\zeta(3)-\frac{1}{2}\zeta(3)^2-\zeta(5)\left(1+\frac{\pi^2}{6}\right)$&1.047657410\\\hline
$m=7$&$\frac{\pi^8}{4200}-\frac{\pi^6}{945}-\zeta(5)\zeta(3)+4\zeta(7)-\frac{\pi^2}{6}\zeta(5)-\frac{\pi^4}{90}\zeta(3)$&1.022090029\\\hline
$m=8$&$\frac{\pi^8}{4200}+5\zeta(9)-\frac{\pi^6}{945}\zeta(3)-\frac{\pi^4}{90}\zeta(5)-\zeta(5)\zeta(3)-\zeta(7)\left(1+\frac{\pi^2}{6}\right)$&1.010557246\\\hline
$m=9$&$\frac{\pi^{10}}{34020}-\frac{\pi^8}{9450}-\zeta(7)\zeta(3)-\frac{1}{2}\zeta(5)^2+5\zeta(9)-\frac{\pi^2}{6}\zeta(7)-\frac{\pi^6}{945}\zeta(3)-\frac{\pi^4}{90}\zeta(5)$&1.005133570\\\hline
$m=10$&$\frac{\pi^{10}}{34020}+6\zeta(11)-\frac{\pi^8}{9450}\zeta(3)-\frac{\pi^6}{945}\zeta(5)-\frac{1}{2}\zeta(5)^2-\frac{\pi^4}{90}\zeta(7)-\zeta(7)\zeta(3)\quad-\zeta(9)\left(1+\frac{\pi^2}{6}\right)$&1.002522063\\\hline
\end{longtable}
\end{center}

\vspace{5mm}
\begin{center}
$S(3,m)$\\
\begin{longtable}{|l|p{60mm}|l|l|}
\hline Power of $n$&Closed form&Approx. value\\\hline
$m=4$&$\frac{\pi^4}{48}-\frac{\pi^2}{8}-\frac{\pi^2}{6}\zeta(3)-\frac{1}{4}\zeta(3)+3\zeta(5)$&1.628620203\\\hline
$m=5$&$\frac{\pi^6}{540}-\frac{\pi^4}{144}-\frac{\pi^2}{4}\zeta(3)-\frac{3}{4}\zeta(3)-\frac{1}{2}\zeta(3)^2+\frac{9}{2}\zeta(5)$&1.180103635\\\hline
$m=6$&$\frac{\pi^6}{360}-\frac{\pi^4}{120}+4\zeta(7)-\frac{\pi^2}{6}\zeta(5)-\frac{\pi^4}{90}\zeta(3)-\frac{3}{4}\zeta(3)^2+\frac{1}{4}\zeta(5)-\frac{\pi^2}{12}\zeta(3)$&1.072362484\\\hline
$m=7$&$\frac{\pi^8}{4200}-\frac{\pi^6}{2520}-\frac{\pi^4}{60}\zeta(3)-\frac{1}{4}\zeta(3)^2-\zeta(5)\zeta(3)-\zeta(5)\left(\frac{\pi^2}{4}+\frac{3}{4}\right)+6\zeta(7)$&1.032351029\\\hline
$m=8$&$\frac{\pi^8}{2800}-\frac{\pi^6}{1260}-\zeta(3)\left(\frac{\pi^4}{180}+\frac{\pi^6}{945}\right)-\zeta(5)\left(\frac{\pi^2}{12}+\frac{\pi^4}{90}\right)-\frac{3}{2}\zeta(5)\zeta(3)\quad+\zeta(7)\left(\frac{3}{4}-\frac{\pi^2}{6}\right)+5\zeta(9)$&1.015179175\\\hline
\end{longtable}
\end{center}

\vspace{5mm}
\begin{center}
$S(4,m)$\\
\begin{longtable}{|l|p{60mm}|l|l|}
\hline Power of $n$&Closed form&Approx. value\\\hline
$m=5$&$\frac{\pi^6}{540}-\frac{\pi^4}{810}-\frac{11\pi^2}{216}-\zeta(3)-\frac{11\pi^2}{36}\zeta(3)-\frac{1}{2}\zeta(3)^2+\frac{11}{2}\zeta(5)$&1.310990854\\\hline
$m=6$&$\frac{11\pi^6}{3240}-\frac{\pi^4}{80}-\zeta(3)\left(\frac{\pi^4}{90}+\frac{1\pi^2}{6}+\frac{11}{36}\right)++\frac{11}{12}\zeta(3)^2+\zeta(5)\left(\frac{59}{36}-\frac{\pi^2}{6}\right)+4\zeta(7)$&1.103348021\\\hline
$m=7$&$\frac{\pi^8}{4200}-\frac{11\pi^4}{3240}+\frac{\pi^6}{2430}\quad-\frac{1}{2}\zeta(3)^2-\zeta(3)\left(\frac{11\pi^4}{540}-\frac{\pi^2}{36}\right)\quad-\zeta(3)\zeta(5)-\zeta(5)\left(\frac{5}{6}-\frac{11\pi^2}{36}\right)+\frac{22}{3}\zeta(7)$&1.043816710\\\hline
$m=8$&$\frac{11\pi^8}{25200}-\frac{5\pi^6}{4536}-\zeta(3)\left(\frac{\pi^6}{945}+\frac{\pi^4}{90}\right)-\frac{1}{12}\zeta(3)^2-\frac{11}{6}\zeta(3)\zeta(5)\quad\quad\quad\quad-\zeta(5)\left(\frac{\pi^4}{90}+\frac{\pi^2}{6}+\frac{11}{36}\right)\quad\quad\quad\quad\quad+\zeta(7)\left(\frac{95}{36}-\frac{\pi^2}{6}\right)+5\zeta(9)$&1.020093103\\\hline
\end{longtable}
\end{center}

\end{document}